\documentclass[12pt]{amsproc}

\newtheorem{theorem}{\sc Theorem}[section]
\newtheorem{lemma}[theorem]{\sc Lemma}
\newtheorem{proposition}[theorem]{\sc Proposition}

\theoremstyle{remark}

\numberwithin{equation}{section}

\newcommand{\UU}{\mathcal U}
\newcommand{\GG}{\mathcal G}
\newcommand{\N}{\mathbb{N}}
\usepackage{color}

\begin{document}
\title[Words of Engel type]{Words of Engel type are concise in residually finite groups}
\author{Eloisa Detomi}
\address{Dipartimento di Matematica, Universit\`a di Padova\\
 Via Trieste 63\\ 35121 Padova \\ Italy}
\email{detomi@math.unipd.it}
\author{Marta Morigi}
\address{Dipartimento di Matematica, Universit\`a di Bologna\\
Piazza di Porta San Donato 5 \\ 40126 Bologna \\ Italy}
\email{marta.morigi@unibo.it}
\author{Pavel Shumyatsky}
\address{Department of Mathematics, University of Brasilia\\
Brasilia-DF \\ 70910-900 Brazil}
\email{pavel@unb.br}
\thanks{This research was partially supported by Universit\`a 
di Padova (Progetto di Ricerca di Ateneo:``Invariable generation of groups"). The second author was also supported by GNSAGA (INDAM), 
and the third author by  FAPDF and CNPq.}

\begin{abstract} Given a group-word $w$ and a group $G$, the verbal subgroup
$w(G)$ is the one generated by all $w$-values in $G$. The word $w$ is said to be concise if $w(G)$ is finite whenever  the set 
of $w$-values in $G$ is finite. In the sixties P. Hall asked whether every word is concise but later Ivanov answered this question in the negative. 
On the other hand, Hall's question remains wide open in the class of residually finite groups. In the present article we show that various generalizations of the Engel word are concise in residually finite groups.
\end{abstract}
\maketitle
\section{Introduction} 

Let $w=w(x_1,\dots,x_k)$ be a group-word, that is, a nontrivial element of the free group $F$ with free generators $x_1,x_2,\dots$. The verbal 
subgroup $w(G)$ of a group $G$ determined by $w$ is the subgroup generated by the set $G_w$ consisting of all values $w(g_1,\ldots,g_k)$, 
where $g_1,\ldots,g_k$ are elements of $G$.  A word $w$ is said to be concise if whenever $G_w$ is finite for a group $G$, it always follows 
that $w(G)$ is finite. More generally, a word $w$ is said to be concise in a class of groups $\mathcal X$ if whenever $G_w$ is finite for a group $G\in\mathcal X$, it always follows that $w(G)$ is finite. P. Hall asked whether every word is concise, but later Ivanov proved that this problem has a negative solution in its general form \cite{ivanov} (see also \cite[p.\ 439]{ols}). On the other hand, many relevant words are known to be concise. In particular, it was shown in \cite{jwilson} that the multilinear commutator words are concise. Such words are also known under the name of outer commutator words and are precisely the words that can be written in the form of multilinear Lie monomials. For example the word $[[x_1,x_2,x_3],[x_4,x_5]]$ is a multilinear commutator. Merzlyakov showed that every word is concise in the class of linear groups \cite{merzlyakov} while Turner-Smith proved that every word is concise in the class of  residually finite groups all of whose quotients are again residually finite \cite{TS2}. There is an open problem whether every word is concise in the class of residually finite groups 
(cf. Segal \cite[p.\ 15]{Segal} or Jaikin-Zapirain \cite{jaikin}). In recent years several words were shown to be concise in residually finite groups while their conciseness in the class of all groups remains unknown. 
In particular, it was shown in \cite{as} that if $w$ is a multilinear commutator word and $n$ is a prime-power, then the word $w^n$ is concise
in the class of residually finite groups. Further examples of words that are concise in residually finite groups were discovered in \cite{gushu}.

In the present article we deal with words of Engel type. Set $[x,{}_1y]=[x,y]=x^{-1}y^{-1}xy$ and $[x,{}_{i+1}y]=[[x,{}_{i}y],y]$ for $i\geq1$. The word $[x,{}_ny]$ is called the $n$th Engel word. Due to \cite{fmt} and, independently, \cite{ar} we know that the $n$th Engel word is concise whenever $n\leq4$. It is still unknown whether the $n$th Engel word is concise in the case where $n\geq5$. In the present article we show, among other things, that the Engel words are concise in residually finite groups. This is an immediate consequence of the following theorem.

\begin{theorem}\label{ra} Suppose that $w=w(x_1,\ldots,x_k)$ is a multilinear commutator word. For any $n\geq1$ the word $[w,{}_ny]$ is concise in residually finite groups.
\end{theorem}

Recall that a word $w$ is a law in a group $G$ if $w(G)=1$. A word $w$ is said to imply virtual nilpotency if every finitely generated metabelian group where $w$ is a law has a nilpotent subgroup of finite index. Such words admit several important characterizations (see \cite{black,bume,groves}). It follows from Gruenberg's result \cite{gruen} that the Engel words imply virtual nilpotency. A word $w$ is boundedly concise in a class of groups $\mathcal X$ if for every integer $m$ there exists a number $\nu=\nu(\mathcal X,w,m)$ such that whenever $|G_w|\leq m$ for a group $G\in\mathcal X$ it always follows that $|w(G)|\leq\nu$. Fern\'andez-Alcober and Morigi \cite{fernandez-morigi} showed that every word which is concise in the class of all groups is actually boundedly concise. However it is unclear whether every word which is concise in residually finite groups is boundedly concise. Our next theorem deals with this question for words implying virtual nilpotency.

\begin{theorem}\label{rb} Words implying virtual nilpotency are boundedly concise in residually finite groups.
\end{theorem}

Throughout the article we use the left-normed notation for commutators, that is: $[b_1,b_2,b_3,\ldots,b_k]=[\ldots[[b_1,b_2],b_3],\ldots,b_k]$. Recall that the lower central word $[x_1,\ldots,x_k]$ is usually denoted by $\gamma_{k}$. The corresponding verbal subgroup $\gamma_k(G)$ is the familiar $k$th term of the lower central series of the group $G$. It seems that in the context of the present article the words $\gamma_{k}$ are the most tractable. 

\begin{theorem}\label{rc} Let $k$, $n$ and $q$ be positive integers and let $w$ be the word $[x_1,\ldots,x_k]^q$. Both words $[y,{}_nw]$ and $[w,{}_ny]$ are boundedly concise in residually finite groups.
\end{theorem}

It was conjectured in \cite{gustavoshu} that each word that is concise in residually finite groups is boundedly concise. As of now, the only words for which the conjecture has not yet been confirmed are those that are dealt with in Theorem \ref{ra}. In the case where $w=[[x_1,x_2],[x_3,x_4]]$ is the metabelian word an ad-hoc argument enables us to show that the word $[w,{}_ny]$ is boundedly concise in residually finite groups.

\begin{theorem}\label{rd} Let $w=[[x_1,x_2],[x_3,x_4]]$. For any $n\geq1$ the word $[w,{}_ny]$ is boundedly concise in residually finite groups.
\end{theorem}

It can be easily seen that the problem on conciseness of words in residually finite groups is equivalent to the same problem in profinite groups. On the other hand, perhaps the very concept of conciseness in profinite groups can be broadened. It was conjectured in \cite{dms} that if $w$ is a word and $G$ a profinite group such that the set $G_w$ is countable, then $w(G)$ is finite. The conjecture was confirmed for various words $w$ (see \cite{dms} for details) but not for Engel words. In view of the results obtained in the present article the following question seems to be of interest.
\bigskip

{\it Let $w$ be the $n$-Engel word with $n\geq2$ and suppose that $G$ is a profinite group with only countably many $w$-values. Is $w(G)$ necessarily finite?}
\bigskip

In the next section we describe some important tools developed in the context of the restricted Burnside problem. Section 3 is a collection of mostly well-known facts which are used in the proofs of our main results. Theorem \ref{ra} is proved in Section 4. Theorems \ref{rc} and \ref{rd} are proved in Section 5. Finally, Section 6 is devoted to the proof of Theorem \ref{rb}.

\section{On Engel groups and the restricted Burnside problem}

A variety is a class of groups defined by equations. More precisely, if $W$ is a set of words, the class of all groups satisfying the laws $W\equiv 1$ is called the variety determined by $W$. By a well-known theorem of Birkhoff \cite[2.3.5]{Rob}, varieties are precisely classes of groups closed with respect to taking subgroups, quotients and Cartesian products of their members. Some interesting varieties of groups have been discovered in the context of the 
restricted Burnside problem, solved in the affirmative by Zelmanov \cite{ze1,ze2}.

It is well-known that the solution of the restricted Burnside problem is equivalent to each of the following statements.

\begin{itemize}
\item[$(i)$] The class of locally finite groups of exponent $n$ is a variety.
\item[$(ii)$] The class of locally nilpotent groups of exponent $n$ is a variety.
\end{itemize}

Recall that a group is said to locally have some property if all its finitely generated subgroups have that property. A group $G$ is of finite exponent $n$ if and only if each element of $G$ has order dividing $n$. A number of varieties of (locally nilpotent)-by-soluble groups were presented in \cite{shu1, shu2}. 

The solution of the restricted Burnside problem strongly impacted our understanding of Engel groups. An element $x\in G$ is called a (left) Engel element if for any $g\in G$ there exists $n=n(x,g)\geq 1$ such that $[g,_n x]=1$. If $n$ can be chosen independently of $g$, then $x$ is a (left) $n$-Engel element. A group $G$ is called $n$-Engel if all elements of $G$ are $n$-Engel. In \cite{ze0} Zelmanov made a remark that the eventual solution of the 
restricted Burnside problem would imply that the class of locally nilpotent $n$-Engel groups is a variety (see also Wilson \cite{w}). It follows that if $G$ is a finite $n$-Engel $d$-generator group, then $G$ is nilpotent with $(d,n)$-bounded nilpotency class. Here and throughout the article we use the expression ``$(a,b,\dots)$-bounded" to mean that a quantity is bounded by a certain number depending only on the parameters $a,b,\dots$. The interested reader is refered to the survey \cite{trau} and references therein for further results on finite and residually finite 
Engel groups. Groups with $n$-Engel word-values were considered in \cite{BSTT,var}. In particular, we will require the following result from \cite{var}.

\begin{theorem}\label{tota} Let $m,n$ be positive integers, and $w$ a multilinear commutator word. The class of all groups $G$ in which the $w^m$-values are $n$-Engel and the verbal subgroup $w^m(G)$ is locally nilpotent is a variety.
\end{theorem}

The following proposition is a consequence of Theorem \ref{tota}.

\begin{proposition}\label{var22} Given positive integers $d,m,n$ and a multilinear commutator word $w$, let $G$ be a finite group in which the $w^m$-values are $n$-Engel. Suppose that a subgroup $H$ can be generated by $d$ elements which are $w^m$-values. Then $H$ is nilpotent with $(d,m,n,w)$-bounded class.
\end{proposition}

Proposition \ref{var22} can be deduced from Theorem \ref{tota} using standard arguments. Indeed, assume that the proposition is false. Then there is an infinite sequence of finite groups $G_1,G_2,\dots$ with subgroups $H_i\leq G_i$ satisfying the hypotheses of the proposition and having nilpotency class tending to infinity. Each subgroup $H_i$ is generated by $d$ elements which are $w^m$-values, say $a_{i1},\dots,a_{id}$. Now let $G$ be the Cartesian product of the groups $G_i$ and consider the subgroup $H$ of $G$ generated by $d$ elements $a_j=(a_{1j},a_{2j},\dots,a_{ij},\dots)$ with $j=1,\dots,d$. On the one hand, by Theorem \ref{tota} the verbal subgroup $w^m(G)$ is locally nilpotent. On the other hand, we see that $H$ is generated by finitely many $w^m$-values in $G$ and is not nilpotent. This  contradiction completes the proof of Proposition \ref{var22}.

We will also require a non-quantitative version of this result. A proof of the next proposition can be found in \cite{BSTT}.

\begin{proposition}\label{non22} Given positive integers $m,n$ and a multilinear commutator word $w$, let $G$ be a residually finite group in which the $w^m$-values are $n$-Engel. Suppose that a subgroup $H$ can be generated by finitely many Engel elements. Then $H$ is nilpotent.
\end{proposition}

The proofs of all results mentioned in this section are based on Lie-theoretic techniques created by Zelmanov in his solution of the 
restricted Burnside problem. This is also the case with the proofs of the results obtained in the present article. In particular, the following theorem plays a fundamental role in all our arguments.

\begin{theorem}\label{zeze}
Let $L=\langle a_1,\dots,a_d\rangle$ be a finitely generated Lie algebra satisfying a polynomial identity. Assume that all commutators in the generators $a_1,\dots,a_d$ are ad-nilpotent. Then $L$ is nilpotent.
\end{theorem}

Theorem \ref{zeze} represents the most general form of the Lie-theoretic part of the solution of the restricted Burnside problem. It was announced in \cite{zelm} with only a sketch of proof. A detailed proof was recently published in \cite{ze5}.

\section{Preliminaries}

Throughout the article we denote by $G'$ the commutator subgroup of a group $G$ and by $\langle M\rangle$ the subgroup generated by a subset $M\subseteq G$. As usual, $C_G(M)$ denotes the centralizer of a subset $M$ in $G$. The preparatory work for the proof of our main theorems requires some preliminary results. We start this section with some well-known facts concerning Engel elements in groups.

\begin{lemma}\label{fitting} Let $G$ be a group generated by (left) Engel elements.
\begin{enumerate}
\item If $G$ is finite, then it is nilpotent.
\item If $G$ is soluble, then it is locally nilpotent.
\end{enumerate}
\end{lemma}
\begin{proof} The first statement is well-known Baer's theorem (see \cite[Satz~III.6.15]{hup}). The second statement is due to Gruenberg \cite{gruen}.
\end{proof}

An element $g\in G$ is a right Engel element if for each $x\in G$ there exists a positive integer $n$ such that $[g,{}_nx]=1$. If $n$ can be chosen independently of $x$, then $g$ is a right $n$-Engel element. The next observation is due to Heineken (see \cite[12.3.1]{Rob}).

\begin{lemma}\label{heineken} Let $g$ be a right $n$-Engel element in a group $G$. Then $g^{-1}$ is a left $(n+1)$-Engel element.
\end{lemma}

 We will also require the following lemmas.

\begin{lemma}[{\cite[Lemma 2.2]{fernandez-morigi}}] \label{inverse} Let $w$ be a multilinear commutator word. Then $G_w$ is symmetric, that is, 
$x\in G_w$
implies that $x^{-1}\in G_w$.
\end{lemma} 

\begin{lemma}\label{zero} Let $w$ be a word and $G$ a group such that the set of $w$-values in $G$ is finite with at most $m$ elements. Then the order of the commutator subgroup $w(G)'$ is $m$-bounded.
\end{lemma}

\begin{proof} The group $G$ acts on the set of $w$-values by conjugation and therefore $G/C_G(w(G))$ embeds in the symmetric group on $m$ symbols. It follows that the order of $w(G)/Z(w(G))$ is at most $m!$ and the result follows from Shur's Theorem (see \cite[p. 102]{Rob2}). 
\end{proof}

\begin{lemma}\label{normal closure}
Let $G=\langle g,t\rangle$ be a group such that $[g,{}_nt]=1$. Then the normal closure of the subgroup $\langle g\rangle$ in the group $G$ is generated by the set $\{g^{t^i}|t=0,\dots, n-1\}$.
\end{lemma}
\begin{proof} It is enough to note that $\langle g, g^t,\dots ,g^{t^j}\rangle=\langle g,[g,t],\dots,[g,{}_jt]\rangle$ for each natural number $j$.
\end{proof}
The following lemma is taken from \cite{danilo}.
\begin{lemma}\label{danilo} Let G be a group generated by $m$ elements which are $n$-Engel. Suppose that $G$ is soluble with derived length $s$. Then $G$ is nilpotent with $(m,n,s)$-bounded class.
\end{lemma}
\begin{lemma}\label{abel}
Let $G$ be a group. Assume that $A$ is a normal abelian subgroup of $G$ and let $t\in G$. Then $[ab,{}_nt]=[a,{}_nt][b,{}_nt]$ for every $a,b\in A$ and for every $n\ge 1$.
\end{lemma}
\begin{proof} This is an easy induction on $n$, using the well-known commutator identity $[xy,z]=[x,z]^y[y,z]$.
\end{proof}

For a subgroup $A$ of a group $G$, an element  $x \in G$ and a positive  integer $n$, we write $[A,{}_n x]$ to denote the subgroup generated by all elements $[a,{}_n x]$, with $a\in A$. The following lemma is due to Casolo.

\begin{lemma}[{\cite[Lemma 6]{Casolo}}] \label{casolo}  Let $A$ be an abelian group, and let $x$ be an automorphism of $A$ such that $[A,{}_nx]=1$ for some $n\ge1$. If $x$ has finite order $q$, then $[A^{q^{n-1}},x]=1$.
\end{lemma} 

\begin{lemma}\label{criterion for nilpotency} Let $G=U\langle t\rangle$ be a group that is a product of a normal subgroup $U$ and a cyclic subgroup $\langle t\rangle$. Assume that  $U$ is nilpotent of class $c$ and there exists a generating set $A$ of $U$ such that $[a,{}_nt]=1$ for every $a\in A$. Then $G$ is nilpotent of $(c,n)$-bounded class.
\end{lemma}
\begin{proof} The proof is by induction on the nilpotency class of $U$. If $U$ is abelian, we have $[uv,{}_nt]=[u,{}_nt][v,{}_nt]$ for each $u,v\in U$. Since $U=\langle A\rangle$, we conclude that $[U,{}_nt]=1$ and so $G$ is nilpotent with class at most $n$.

If $U$ is not abelian, note that $G/U'$ is nilpotent with class at most $n$.
Taking into account that $U$ is nilpotent of class $c$, the result follows from P. Hall's criterion for nilpotency \cite[5.2.10]{Rob}. 
\end{proof}

\section{Theorem \ref{ra}}

Let $F$ be the free group, and let $F_i$ denote the $i$th term of the lower central series of $F$. We say that a word $w$ has degree $j$ if $w\in F_j$ and $w\not\in F_{j+1}$.

\begin{lemma}\label{Pavel} Let $p$ be a prime and let $w=w(x_1,x_2,\dots,x_k)$ be a word of degree $j$. Let $G=\langle a_1,a_2,\dots,a_k\rangle$ be a nilpotent group of class $c$ generated by $k$ elements $a_1,a_2,\dots,a_k$. Denote by $X$ the set of all conjugates in $G$ of elements of the form $w(a_1^i,a_2^i,\dots,a_k^i)$, where $i$ is an integer not divisible by $p$, and assume that $|X|\leq m$ for some integer $m$. Then $|\langle X\rangle|$ is $(c,m)$-bounded.
\end{lemma}
\begin{proof} Let $W=\langle X\rangle$. As in Lemma \ref{zero}, the order of $W'$ is $m$-bounded. Thus we can pass to the quotient $G/W'$  
 and assume that $W$ is abelian.

If $j\geq c+1$, then $W=1$ and there is nothing to prove. Therefore we assume that $j\leq c$ and use induction on $c-j$. In the free group, modulo $F_{j+1}$, the word $w$ is a product of 
$\gamma_j$-words in $x_1,\dots,x_k$. Therefore for any $s$ we have $$w(x_1^s,x_2^s,\dots,x_k^s)=w(x_1,x_2,\dots,x_k)^{s^j}w_s(x_1,x_2,\dots,x_k),$$ where $w_s(x_1,x_2,\dots,x_k)$ is a word of degree at least $j+1$ (see for example \cite[Lemma 1.3.3]{Segal}). Here the word $w_s$ depends on $s$. For this word, if $s$ is chosen coprime to $p$, there are at most $m^2$ conjugates of elements of the form $w_s(a_1^i,a_2^i,\dots,a_k^i)$, where $(i,p)=1$. Indeed, we have $$w_s(a_1^i,a_2^i,\dots,a_k^i)=(w(a_1^i,a_2^i,\dots,a_k^i)^{s^j})^{-1}w(a_1^{is},a_2^{is},\dots,a_k^{is}).$$ In the group $G$ there are at most $m$ conjugates of elements of the form $w(a_1^i,a_2^i,\dots,a_k^i)$ and as many of the form $w(a_1^{is},a_2^{is},\dots,a_k^{is})$. Hence, there are at most $m^2$ conjugates of elements of the form $w_s(a_1^i,a_2^i,\dots,a_k^i)$.

Since $w_s\in F_{j+1}$, we can use induction and conclude that the subgroup
generated by the conjugates in $G$ of elements of the form $w_s(a_1^i,a_2^i,\dots,a_k^i)$, where $(i,p)=1$, has $(c,m)$-bounded order $B$. We emphasize that $B$ does not depend on $s$. Recall that $W$ is an abelian subgroup with at most $m$ generators. The product $U$ of all subgroups of order at most $B$ contained in $W$ has order bounded in terms of $B$ and $m$ only. Thus, we can pass to the quotient over $U$ and assume that $$w(a_1^s,a_2^s,\dots,a_k^s)=w(a_1,a_2,\dots,a_k)^{s^j}$$ whenever $(s,p)=1$.

Here, by the hypothesis, the left-hand side of the equality can take at most $m$ different values while $s$ can be any integer coprime to $p$. Note that the number of positive integers coprime to $p$ and smaller than or equal to $2m+1$ is bigger than $m$. Thus, there exist two integers $s_1$ and $s_2$, smaller than or equal to $2m+1$, such that $$w(a_1,a_2,\dots,a_k)^{s_1^j}=w(a_1,a_2,\dots,a_k)^{s_2^j}.$$ In particular, the order of $w(a_1,a_2,\dots,a_k)$ is at most $(2m+1)^j$. We conclude that $w(a_1,a_2,\dots,a_k)$ has finite $(c,m)$-bounded order. The equality $w(a_1^s,a_2^s,\dots,a_k^s)=w(a_1,a_2,\dots,a_k)^{s^j}$ further shows that for each $s$ the element $w(a_1^s,a_2^s,\dots,a_k^s)$ has order dividing that of $w(a_1,a_2,\dots,a_k)$. Thus, $W$ is an abelian subgroup with at most $m$ generators, each of $(c,m)$-bounded order. Therefore the order of $W$ is $(c,m)$-bounded.
\end{proof}

Important properties of the verbal subgroup corresponding to a multilinear commutator word in a soluble group are presented in the following proposition.

\begin{proposition}[{\cite[Theorem B]{fernandez-morigi}}] \label{series}
Let $w$ be a multilinear commutator word, and let $G$ be  a soluble group.
There exists a series of subgroups $$w(G)=K_1\geq\dots\geq K_l=1$$ such that:
\begin{enumerate}
\item All subgroups $K_i$ are normal in $G$.
\item The length $l$ of the series is bounded in terms of $w$ and the derived length of $G$.
\item Every section $K_i/K_{i+1}$ is abelian and can be generated by $w$-values in $G/K_{i+1}$ all of whose powers are also $w$-values.
\end{enumerate}
\end{proposition}

The next lemma supplies the proof of Theorem \ref{ra} in the particular case where the group $G$ is soluble. The result holds for any $q\ge 1$ but we will need it only in the case $q=1$. Note that $G$ here is not required to be residually finite. 
\begin{lemma}\label{soluble} Let $m,n,q,s$ be positive integers. Suppose that $w=w(x_1,\ldots,x_k)$ is a multilinear commutator word and $v=[w^q,{}_ny]$. Assume that $G$ is a soluble group of derived length $s$ such that $v$ has at most $m$ values in $G$. Then the order of $v(G)$ is $(v,m,s)$-bounded.
\end{lemma}
\begin{proof} In view of Lemma \ref{zero} we may assume that $v(G)$ is abelian. Consider a series in $w(G)$ as in Proposition \ref{series}.
We will use induction on the length of this series, the case $w(G)=1$ being trivial. Let $L$ be the last nontrivial term of the series. By induction we assume that the image of $v(G)$ in $G/L$ has finite $(v,m,s)$-bounded order.
Set $K=v(G)\cap L$. It follows that the index of $K$ in $v(G)$ is $(v,m,s)$-bounded.

The subgroup $L$ is generated by $w$-values all of whose powers are $w$-values. Let $g\in L$  be one of those $w$-values. Then for every positive integer $i$ the element $[g^{iq},{}_nt]$ is a $v$-value. As $L$ is abelian, it follows that $[g^{iq},{}_nt]=[g^q,{}_nt]^i$. Therefore there are two different integers $i_1,i_2$ with $0\le i_1<i_2\le m$ such that $[g^q,{}_nt]^{i_1}=[g^q,{}_nt]^{i_2}$. It follows that $[g^q,{}_nt]$ has order at most $m$, and consequently $[g,{}_nt]$ has order at most $mq$. Let $T_1$ be the subgroup of $v(G)$ generated by all elements of order at most $mq$. As $v(G)$ is abelian with at most $m$ generators, $T_1$ is finite with $(m,q)$-bounded oder. Thus, we can pass to the quotient  $G/T_1$ and assume that $[g,{}_nt]=1$ for every generator $g$ of $L$ chosen as above and for every $t\in G$. As $L$ is abelian, it follows that $[L,{}_nt]=1$ for every $t\in G$. In particular, $[K,{}_n t]=1$  for every $t\in G$.

Since the index of $C_G(v(G))$ in $G$ is $m$-bounded, conjugation by an arbitrary element $t\in G$ induces an automorphism of $K$ of $m$-bounded order, say $r$. Lemma \ref{casolo} tells us that $[K^{r^{n-1}},t]=1$. As $K$ is abelian, it follows that $[K,t]$ has exponent dividing $r^{n-1}$, which is $(m,n)$-bounded. Let $T_2$ be the subgroup of $v(G)$ generated by all elements of order at most $r^{n-1}$. We can pass to the quotient $G/T_2$ and without loss of generality assume that $[K,t]=1$ for every $t\in G$. Therefore $K$ is contained in the center of the group $G$. Further, note that $K\langle t^r\rangle$ is a central subgroup of $v(G)\langle t\rangle$ and has $(v,m,s)$-bounded index in $v(G)\langle t\rangle$. So by Shur's Theorem the derived subgroup of $v(G)\langle t\rangle$ has $(v,m,s)$-bounded order, which does not depend on the choice of $t$. Arguing as before and factoring out an appropriate small subgroup of $v(G)$, we may assume that $[v(G),t]=1$ for every $t\in G$, that is, $v(G)$ is contained in the center of $G$.

In particular, $[g^q,{}_{n+1}t]=1$ for every $g\in G_w$ and every $t\in G$. So every $w^q$-value is right $(n+1)$-Engel in $G$. Thus, by Lemma \ref{heineken} combined with Lemma \ref{inverse}, every $w^q$-value is left $(n+2)$-Engel. Lemma \ref{fitting} now says that $w^q(G)$ is locally nilpotent.

Choose again $g\in G_{w}$ and $t\in G$. It follows from Lemma \ref{normal closure} that the normal closure $U$ of the subgroup $\langle g^q\rangle$
in the group $\langle g,t\rangle$ is generated by the set $A=\{(g^q)^{t^i}|i=0,\dots, n\}$ whose elements are left $(n+2)$-Engel. Lemma \ref{danilo} now tells us that $U$ is nilpotent with $(n,s)$-bounded class. As $[a,{}_{n+1}t]=1$ for every $a\in A$, Lemma \ref{criterion for nilpotency} shows that  $U\langle t\rangle$ is nilpotent of $(v,m,s)$-bounded class.

Hence, we are in a position where Lemma \ref{Pavel} can be applied (with just an arbitrary $p$). Note that $$[g^q,{}_{n}t]=[(t^{-1})^{g^qt},{}_{n-1} t]=[t^{-g^q},{}_{n-1}t]^{t},$$ and so $[g^q,{}_{n}t]$ is a value of the word
$\tilde v=[x_1,_{n-1}x_2]$ in the subgroup $\langle t^{-g^q},t\rangle$ of $U\langle t\rangle$. As $\tilde v((t^{-g^q})^i,t^i)=[g^q,{}_{n}t^i]^{t^{-i}}\in G_v$, the set of all conjugates in $\langle t^{-g^q},t\rangle$ of elements of the form $\tilde v((t^{-g^q})^i,t^i)$, where $i$ is an integer, has at most $m$ elements. Therefore it follows from Lemma \ref{Pavel} that the cyclic subgroup generated by $[g^q,{}_{n}t]$ has $(v,m,s)$-bounded order. 

Thus, we have shown that $v(G)$ is an abelian group of rank at most $m$ generated by elements of $(v,m,s)$-bounded order. Hence, the order of $v(G)$ is  $(v,m,s)$-bounded. The proof is complete.
\end{proof}

An important family of multilinear commutator words is formed by the derived words $\delta_k$, on $2^k$ variables, which are defined recursively by
$$\delta_0=x_1,\qquad \delta_k=[\delta_{k-1}(x_1,\ldots,x_{2^{k-1}}),\delta_{k-1}(x_{2^{k-1}+1},\ldots,x_{2^k})].$$
Of course $\delta_k(G)=G^{(k)}$, the $k$-th derived subgroup of $G$. We will need the following well-known result (see for example \cite[Lemma 4.1]{S2}). 
\begin{lemma}\label{lem:delta_k} Let $G$ be a group and  let $w$ be a multilinear commutator word on $k$ variables. Then each $\delta_k$-value is a $w$-value.
\end{lemma}

Now we are ready to embark on the proof of Theorem \ref{ra}. 

\begin{proof}[Proof of Theorem \ref{ra}] Recall that $w=w(x_1,\ldots,x_k)$ is a multilinear commutator word. We wish to prove that the word $v=[w,{}_ny]$ is concise in residually finite groups. Thus, let $G$ be a residually finite group in which the word $v$ has only finitely many values. We need to show that $v(G)$ is finite.

In view of Lemma \ref{zero} we may assume that $v(G)$ is abelian. Since $v(G)$ is finitely generated, it is clear that elements of finite order 
in $v(G)$ form a finite normal subgroup. We pass to the quotient over this subgroup and without loss of generality assume that $v(G)$ is torsion-free.
Let $\bar G$ be any finite quotient of $G$. Since $v(\bar G)$ is abelian and the image of $w(\bar G)$ in $\bar G/v(\bar G)$ consists of right Engel elements, we conclude that $\bar G$ is soluble. Thus, our group $G$ is residually soluble. Taking into account that $C_G(v(G))$ has finite index in $G$, we deduce that some term of the derived series of $G$, say $G^{(l)}$, is contained in $C_G(v(G))$. Let $j$ be the maximum of the numbers $k$ and $l$. Then $G^{(j)}$ centralizes $v(G)$ and, by Lemma \ref{lem:delta_k}, every $\delta_j$-value is also a $w$-value. 

Let $x,y\in G$ with $y$ being a $\delta_j$-value. Using the formula $[x,y,y]=[y^{-x},y]^y$ and taking into account that $y^{-1}$ is a $w$-value (Lemma \ref{inverse}), we deduce that the commutator $[x,{}_{n+1}y]$ is a $v$-value and therefore, since $[v(G),y]=1$, we have $[x,{}_{n+2}y]=1$. We just have shown that each $\delta_j$-value is $(n+2)$-Engel in $G$. In view of Proposition \ref{non22} any subgroup of $G$ generated by finitely many $\delta_j$-values is nilpotent and so $G^{(j)}$ is locally nilpotent. 

Choose $\delta_{2j}$-values $y_1,\dots,y_f\in G$ and for each $i=1,\dots,f$ write $y_i=\delta_k(g_{i1},\dots,g_{i2^{j}})$, where each $g_{ii'}$ is a $\delta_j$-value. Further, choose an arbitrary element $t\in G$ and let $H$ be the minimal $t$-invariant subgroup of $G$ containing all these elements $g_{ii'}$. Since the image of $t$ in $G/v(G)$ acts on each $g_{ii'}$ as an Engel element, Lemma \ref{normal closure} tells us that the image of $H$ in $G/v(G)$ is generated by finitely many $\delta_j$-values. Therefore the image of $H$ is nilpotent and, since $v(G)$ is abelian, we conclude that $H\langle t\rangle$ is soluble. Now Proposition \ref{soluble} tells us that $v(H\langle t\rangle)$ is finite. Since $v(G)$ is torsion-free, $v(H\langle t\rangle)=1$. In particular, $[y_i,{}_nt]=1$ for each $i=1,\dots,f$. We now invoke Lemma \ref{criterion for nilpotency} and conclude that $\langle t,y_1,\dots,y_f\rangle$ is nilpotent. This happens for any choice of $\delta_{2j}$-values $y_1,\dots,y_f\in G$ and $t\in G$. Therefore $\langle G^{(2j)},t\rangle$ is locally nilpotent for each $t\in G$.

As $v(G)\cap G^{(2j)}$ is a finitely generated abelian group, for each $t\in G$ there exists an integer $s$ such that $[v(G)\cap G^{(2j)},{}_st]=1$. Since the index of $C_G(v(G))$ in $G$ is finite, the conjugation by $t$ is an automorphism of $v(G)\cap G^{(2j)}$ of finite order, say $r$. In view of  Lemma \ref{casolo} we obtain that $[(v(G)\cap G^{(2j)})^{r^{s-1}},t]=1$. As $v(G)\cap G^{(2j)}$ is abelian, it follows that $[v(G)\cap G^{(2j)},t]$ has finite exponent at most $r^{s-1}$. Taking into account that $v(G)$ is torsion-free, we conclude that $[v(G)\cap G^{(2j)},t]=1$ for every $t\in G$. Therefore $v(G)\cap G^{(2j)}$ is contained in the center of the group $G$. 

Since $G/G^{(2j)}$ is soluble, Proposition \ref{soluble} guarantees that the image of $v(G)$ in $G/G^{(2j)}$ is finite. In other words, $v(G)\cap G^{(2j)}$ has finite index  in $v(G)$. Note that $(v(G)\cap G^{(2j)})\langle t^{r}\rangle$ is a central subgroup of  finite index in $v(G)\langle t\rangle$. So by Schur's Theorem the commutator subgroup of $v(G)\langle t\rangle$ is finite. Since the commutator subgroup is contained in $v(G)$, which is torsion-free, it follows that $[v(G), t]=1$ for every $t\in G$, that is, $v(G)$ is contained in the center of $G$. In particular, $[g,{}_{n+1}t]=1$ for every $g\in G_w$ and every $t\in G$. Thus, every $w$-value is right $(n+1)$-Engel in $G$, whence by Lemma \ref{heineken} combined with Lemma \ref{inverse} every $w$-value is left $(n+2)$-Engel. 

Let $g\in G_{w}$ and $t\in G$. It follows from Lemma \ref{normal closure} that the normal closure $U$ of the subgroup $\langle g\rangle$ in the group $\langle g,t\rangle$ is generated by the set $A=\{(g)^{t^i}|t=0,\dots,n\}$ whose elements are left $(n+2)$-Engel. Therefore, by Proposition \ref{non22}, $U$ is nilpotent. Hence, by Lemma \ref{criterion for nilpotency} the subgroup $\langle g,t\rangle$ is nilpotent, too. 

Now we are in a situation where Lemma \ref{Pavel} can be applied. Note that $[g,{}_{n}t]=[(t^{-1})^{gt},{}_{n-1}t]=[t^{-g},{}_{n-1}t]^{t}$. Thus $[g,{}_{n}t]$ is a value of the word $\tilde v=[x_1,_{n-1}x_2]$ in the group $\langle t^{-g},t\rangle$. Since $\tilde v((t^{-g})^i,t^i)=[g,{}_{n}t^i]^{t^{-i}}\in G_v$, the set of all conjugates in $\langle t^{-g},t\rangle$ of elements of the form $\tilde v((t^{-g})^i,t^i)$, where $i$ is an integer, is finite. So it follows from Lemma \ref{Pavel} that $[g,{}_{n}t]$ has finite order. 

Thus, an arbitrary $v$-value in $G$ has finite order. Since $v(G)$ is torsion-free, we conclude that $v(G)=1$. The theorem is established.
\end{proof}

\section{Proofs of Theorems \ref{rc} and \ref{rd}} 

In the present section Theorems \ref{rc} and \ref{rd} will be proved. We start with the general remark that a word $w$ is boundedly concise in residually finite groups if and only if the order of $w(G)$ is bounded in terms of $w$ and $|G_w|$ for each finite group $G$. It follows that Theorems \ref{rc} and \ref{rd} are essentially about finite groups and their proofs will deal with corresponding questions for finite groups.

An important concept required in this section is that of weakly rational words. Following \cite{gushu} we say that a word $w$ is weakly rational if for every finite group $G$ and for every integer $e$ relatively prime to $|G|$, the set $G_w$ is closed under taking $e$-th powers of its elements. 
By \cite[Lemma 1]{gushu}, the word $w$ is weakly rational if and only if for every finite group $G$ and $g\in G_w$, the power $g^e$ belongs to $G_w$ whenever $e$ is relatively prime to $|g|$. It was shown in \cite[Theorem 3]{gushu} that for every positive integers $k$ and $q$ the word $w=[x_1,\ldots,x_k]^q$ is weakly rational.

Let $w$ be the word $[x_1,\ldots,x_k]^q$. Theorem \ref{rc} states that both words $[y,{}_nw]$ and $[w,{}_ny]$ are boundedly concise in residually finite groups. We will treat the two words separately. Proposition \ref{rational1} deals with the word $[y,{}_n w]$ while Proposition \ref{rational2} with the word $[w,{}_ny]$. 

\begin{proposition}\label{rational1} Let $k, q$ and $n$ be positive integers, and let  $w=[x_1,\ldots,x_k]^q$. The word $[y,{}_n w]$ is boundedly concise in residually finite groups.
\end{proposition}

\begin{proof} Let $G$ be a finite group with at most $m$ values of the word $v=[y,{}_n w]$. It is sufficient to  prove that the order of $v(G)$ is $(k,m,n,q)$-bounded. In view of Lemma \ref{zero} we may assume that $v(G)$ is abelian. Since $v(G)$ is an abelian $m$-generated group, for every integer $r$ the order of the subgroup of $v(G)$ generated by all elements of order at most $r$ is $(r,m)$-bounded. 

Using that $v(G)$ is abelian we deduce that \[[h^i, {}_n g]=[h, {}_n g]^i \] for every integer $i$ and elements $h\in v(G)$, $g\in G$. Note that in the case where $g\in G_w$ we have $[h^i,{}_ng]\in G_v$. Therefore every power of the $v$-value $[h,{}_ng]$ is a $v$-value. Since $G_v$ has at most $m$ elements, it follows that \[[h,{}_ng]^i=[h,{}_ng]^j\] for some $0\le i\neq j \le m$. Hence, the order of $[h,{}_ng]$ is at most $m$. Let $M_1$ be  the subgroup of $v(G)$ generated by all elements of order at most $m$. Since $M_1$ has  $m$-bounded order, we can pass to the quotient $G/M_1$ and thus assume that $[h,{}_n g]=1$ for every $h\in v(G)$ and $g\in G_w$. With this in mind, for any $t\in G$ and $g\in G_w$ we deduce that 
 $$[[t, {}_n g],{}_n g]=[t,{}_{(2n)}g]=1.$$ 
 This means that every $w$-value is $2n$-Engel in $G$. It follows from Lemma \ref{fitting} that $w(G)$ is nilpotent. We will now additionally assume that $w(G)$ is a $p$-group for a prime $p$. 
   
Let us fix a $v$-value $[t,{}_ng]$, where $t\in G$ and $g\in G_w$, and consider the subgroup $H=\langle g^{-t},g\rangle$. Since $G$ is a finite group in which the $w$-values are $2n$-Engel and $H$ is generated by two $w$-values, Proposition \ref{var22} implies that $H$ has $(k,q,n)$-bounded nilpotency class. Further, we know from \cite[Theorem 3]{gushu} that the word $w$ is weakly rational. Hence, $g^i\in G_w$ for every integer $i$ prime to $p$. Note that \[ [t,{}_n g^i]=[(g^i)^{- tg^i}, {}_{(n-1)} g^i]=[(g^i)^{-t},{}_{(n-1)} g^i]^{g^i},\] whence we conclude that $[g^{-i t},{}_{(n-1)} g^i]\in G_v$ for every integer $i$ coprime to $p$. In particular the set 
  \[X=\{[(g^{-t})^i,{}_{(n-1)} g^i]^x\mid x\in G, \ (i,p)=1\}\] 
   is a subset of $G_v$ and therefore $|X|\le m$. 
   
Let $\eta(x_1,x_2)$ denote the $(n-1)$-Engel word. We are in situation where $H=\langle g^{-t},g\rangle$ is nilpotent with $(k,q,n)$-bounded class and the set $$\{\eta(g^{-it},g^i)^x=[(g^{-t})^i,{}_{(n-1)} g^i]^x \mid x\in H, \ (i,p)=1\}$$ is a subset of $X$. Hence it has at most $m$ elements. We deduce from Lemma \ref{Pavel} that the order of the element $[g^{-t},{}_{(n-1)} g]$ is $(k,q,n,m)$-bounded. So the order of the arbitrary $v$-value $[t,{}_ng]$ is bounded by a number which depends only on $k,q,n$ and $m$. Of course this implies that the order of $v(G)$ is $(k,m,n,q)$-bounded.

Thus, in the particular case where $w(G)$ is a $p$-group the proposition is proved. It is important to note that we proved existence of a bound, say $B$, for $|v(G)|$ which does not depend on $p$.

We will now deal with the case where $w(G)$ is not necessarily a $p$-group.
Let $p_1,\dots,p_s$ be the set of prime divisors of the order of $w(G)$. Since the case where $s=1$ was already dealt with, we assume that $s\geq2$.
Recall that $w(G)$ is nilpotent and so any Hall subgroup of $w(G)$ is normal in $G$. For each $i=1,\dots,s$ let $N_i$ denote the Hall ${p_i}'$-subgroup of $w(G)$. The result obtained in the case where $w(G)$ is a $p$-group implies that for any $i$ the image of $v(G)$ in $G/N_i$ has order at most $B$. It follows that $v(G)$ embeds into a direct product of abelian groups of order at most $B$. Therefore the exponent of $v(G)$ divides $B!$. Thus $v(G)$ is an abelian group with $m$ generators and exponent dividing $B!$. We conclude that the order of $v(G)$ is $(k,m,n,q)$-bounded, as required.
  \end{proof}

\begin{lemma}\label{eta} Let $w=w(x_1,\dots,x_k)$ be a word and $n$ a positive integer. There is a word $\eta$ in $k(n+1)$ variables such that
$$[w,{}_ny]=\eta(x_1,x_2,\ldots,x_k,x_1^y,x_2^y,\ldots,x_k^y,\ldots,x_1^{y^n},x_2^{y^n},\ldots,x_k^{y^n}).$$
\end{lemma}
\begin{proof} If $n=1$, we have $[w,y]=w^{-1}w^y$ and so we can take $\eta(x_1,x_2,\dots,x_{2k})=w(x_1,\dots,x_k)^{-1}w(x_{k+1},\dots,x_{2k})$. Now an obvious induction on $n$ completes the proof.
\end{proof}

\begin{proposition}\label{rational2}  Let $k,n,q$ be positive integers and  $w=[x_1,\ldots,x_k]^q$. The word $[w,{}_ny]$ is boundedly concise in residually finite groups.
\end{proposition}

\begin{proof} Let $G$ be a finite group with at most $m$ values of the word $v=[w,{}_ny]$. It is sufficient to prove that the order of $v(G)$ is $(k,m, n,q)$-bounded.

By Lemma \ref{zero} we may assume that $v(G)$ is abelian. Since $v(G)$ is an abelian $m$-generated group, for every integer $r$ the order of the subgroup of $v(G)$ generated by all elements of order at most $r$ is $(r,m)$-bounded.

Using the formula $[t,{}_{(n+1)}g]=[g^{-tg},{}_ng]$ and the fact that, by Lemma \ref{inverse}, $g^{-1}\in G_w$ we observe that the element 
$[t,{}_{(n+1)}g]$ represents a $v$-value for any $g\in G_w$ and $t\in G$. Thus, the word $[x,{}_{(n+1)}w]$ has at most $m$ values in $G$. 
Applying Proposition \ref{rational1} with the word $[x,{}_{(n+1)}w]$, we deduce that the corresponding verbal subgroup has bounded order. 
Passing to the quotient over this subgroup we assume, without loss of generality, that $[t,{}_{(n+1)}g]=1$ for every $t\in G$ and $g\in G_w$. Thus, all $w$ values are $(n+1)$-Engel in $G$. In particular, it follows from Lemma \ref{fitting} that $w(G)$ is nilpotent.

Now fix a $v$-value $[g,{}_{n}t]$ with $t\in G$ and $g\in G_w$. We use Lemma \ref{eta} and write $[g,{}_{n}t]=\eta(g,g^t,\ldots,g^{t^n})$ for an appropriate word $\eta=\eta(x_1,x_2,\ldots,x_{n+1})$. Set $H=\langle g,g^t,\ldots,g^{t^n}\rangle$. Since $G$ is a finite group in which the $w$-values are $(n+1)$-Engel and $H$ is generated by $n+1$ elements which are $w$-values, Proposition \ref{var22} implies that $H$ has $(k,n,q)$-bounded nilpotency class. 

Recall that $w(G)$ is nilpotent. Assume first that $w(G)$ is a $p$-group for a prime $p$. Then $H$ is a $p$-group as well. We know from \cite[Theorem 3]{gushu}, that the word $w$ is weakly rational and so $g^i\in G_w$ for every integer $i$ coprime to $p$. 
  
Since \[\eta(g^i, g^{it},\ldots,g^{it^n})=[g^i,{}_n t],\] it follows that  $\eta(g^i,g^{it},\ldots,g^{it^n})\in G_v$ for every integer $i$ coprime to $p$. In particular the set \[X=\{\eta(g^i,g^{it},\ldots,g^{it^n})^x\mid x\in H,\ (i,p)=1\}\] is a subset of $G_v$ and hence $|X|\le m$. 
   
Since the nilpotency class of $H$ is $(k,n,q)$-bounded, we deduce from Lemma \ref{Pavel} that the order of the element $\eta(g,g^t,\ldots,g^{t^n})$ is $(k,q,n,m)$-bounded.
 Thus, the order of an arbitrary $v$-value $[g,{}_{n}t]$  is bounded by a number which depends only on $k,n,q$ and $m$. Since $v(G)$ is an abelian subgroup generated by $m$ such elements, we conclude that  
the order of $v(G)$ is $(k,m,n,q)$-bounded. Thus, in the particular case where $w(G)$ is a $p$-group the proposition is proved. It is important to note that we proved existence of a bound, say $B$, for $|v(G)|$ which does not depend on $p$.

We will now deal with the case where $w(G)$ is not necessarily a $p$-group.
Let $p_1,\dots,p_s$ be the set of prime divisors of the order of $w(G)$. Since the case where $s=1$ was already dealt with, we assume that $s\geq2$. Recall that $w(G)$ is nilpotent and so any Hall subgroup of $w(G)$ is normal in $G$. For each $i=1,\dots,s$ let $N_i$ denote the Hall ${p_i}'$-subgroup of $w(G)$. The result obtained in the case where $w(G)$ is a $p$-group implies that for any $i$ the image of $v(G)$ in $G/N_i$ has order at most $B$. It follows that $v(G)$ embeds into a direct product of abelian groups of order at most $B$. Therefore the exponent of $v(G)$ divides $B!$. Thus $v(G)$ is an abelian group with $m$ generators and exponent dividing $B!$. We conclude that the order of $v(G)$ is $(k,m,n,q)$-bounded, as required.
\end{proof}

Having completed the proofs of Proposition \ref{rational1} and Proposition \ref{rational2}, we established Theorem \ref{rc}. Now we deal with Theorem \ref{rd}.

\begin{proof}[Proof of Theorem \ref{rd}] 
Let $G$ be a finite group with at most $m$ values of the word $v=[w,{}_ny]$, where  $ w=[[x_1, x_2],[x_3,x_4]]$. It is sufficient to prove that the order of $v(G)$ is $(m,n)$-bounded. As usual, we may assume that $v(G)$ is abelian.

Choose a commutator $d=[d_1,d_2]$ in $G$ and an arbitrary element $t\in G$. We note that 
\[[t,{}_{(n+2)}d]=[d^{-td},{}_{(n+1)}d].\] Since $[d^{-td},d]\in G_w$, we conclude that $[t,{}_{(n+2)}d]\in G_v$. Thus, the values of the word $[x,{}_{(n+2)}[x_1,x_2]]$ are contained in $G_v$ and so there are at most $m$ of them. Proposition \ref{rational1} applied with the word $[x,{}_{(n+2)}[x_1,x_2]]$ enables us to deduce that the corresponding verbal subgroup has bounded order. Passing to the quotient over this subgroup we assume that $[t,{}_{(n+2)}d] =1$ for every $t\in G$ and every commutator $d\in G$.
Since all commutators are Engel, it follows from Lemma \ref{fitting} that the commutator subgroup $G'$ is nilpotent.
  
Now fix a $v$-value $[g,{}_{n}t]$ with $t\in G$ and $g\in G_w$. Write $g=[a,b]$, where $a,b$ are commutators. Lemma \ref{eta} says that there exists a word $\eta$ such that 
 \[[g,{}_{n}t]=[[a,b],{}_nt]=\eta(a,b,a^t,b^t,\ldots,a^{t^n}, b^{t^n}).\] 
Consider the subgroup $H=\langle a,b,a^t,b^t,\ldots,a^{t^n}, b^{t^n}\rangle.$ Since $G$ is a finite group in which commutators are 
$(n+2)$-Engel and $H$ is generated by $2(n+1)$ commutators, Proposition \ref{var22} tells us that $H$ has $n$-bounded nilpotency class. 

Recall that $G'$ is nilpotent. Consider first the case where $G'$ is a $p$-group for a prime $p$. It is well-known that the commutator word $[x_1,x_2]$ is weakly rational (see e.g. \cite[p.45]{Isaacs} or \cite{gushu}). Therefore $a^i,b^i$ are commutators for every integer $i$ coprime to $p$. Since \[\eta(a^i,b^i,a^{it},b^{it},\ldots,a^{it^n}, b^{it^n})=[[a^i,b^i],{}_nt],\] it follows that $\eta(a^i,b^i,a^{it},b^{it},\ldots,a^{it^n},b^{it^n})\in G_v$ for every integer $i$ coprime to $p$. In particular the set 
\[X=\{\eta(a^i,b^i,a^{it},b^{it},\ldots,a^{it^n},b^{it^n})^x\mid x\in H,\ (i,p)=1\}\] is a subset of $G_v$ and hence $|X|\le m$. 
   
Since the nilpotency class of $H$ is $n$-bounded, Lemma \ref{Pavel} guarantees that the order of the element $\eta(a,b,a^t,b^t,\ldots,a^{t^n},b^{t^n})$ is $(n,m)$-bounded. Thus, we have shown that the order of an arbitrary $v$-value in $G$ is bounded by a number depending only on $n$ and $m$. We conclude that the order of $v(G)$ is $(m,n)$-bounded. Therefore in the particular case where $G'$ is a $p$-group the theorem is proved. It is important to note that we proved existence of a bound, say $B$, for $|v(G)|$ which does not depend on $p$.

The case where $G'$ is not necessarily a $p$-group will be dealt with using familiar arguments, similar to those employed in the proof of Theorem \ref{rc}.
Let $p_1,\dots,p_s$ be the set of prime divisors of the order of $G'$. Since the case where $s=1$ was considered in the previous paragraphs, we assume that $s\geq2$. Observe that any Hall subgroup of $G'$ is normal in $G$. For each $i=1,\dots,s$ let $N_i$ denote the Hall ${p_i}'$-subgroup of $G'$. The result obtained in the case where $G'$ is a $p$-group implies that for any $i$ the image of $v(G)$ in $G/N_i$ has order at most $B$. It follows that $v(G)$ embeds into a direct product of abelian groups of order at most $B$. Therefore the exponent of $v(G)$ divides $B!$. Thus $v(G)$ is an abelian group with $m$ generators and exponent dividing $B!$. We conclude that the order of $v(G)$ is $(m,n)$-bounded, as required.
  \end{proof}

\section{Theorem \ref{rb}}

In this section we will prove that all words that imply virtual nilpotency of finitely generated metabelian groups are boundedly concise in the class
of residually finite groups (Theorem \ref{rb}). The class of such words is fairly large. It coincides with that of the words $w$ such that $w$ is not a law in the wreath product $C_n\wr C$ for any $n$ (see \cite{bume}). Here $C_n$ denotes the cyclic group of order $n$ and $C$ the infinite cyclic group. In particular, any word of the form $uv^{-1}$, where $u$ and $v$ are positive words, is of that kind. Other examples include generalizations of the Engel words like $[x^{n_1},y^{n_2},\dots,y^{n_k}]$ for integers $n_1,n_2,\dots,n_k$.

The classical result of Turner-Smith says that every word is concise in the class of groups all of whose quotients are residually finite \cite{TS2}. In particular, every word is concise in the class of virtually nilpotent groups. Combining this with arguments along the lines of \cite[Appendix]{fernandez-morigi} we will establish the following proposition. 

\begin{proposition}\label{quantitative} Let $c,t$ be positive integers and let $\mathcal X$ be the class of groups having a normal subgroup of finite index at most $t$ which is nilpotent of class at most $c$. 
Then every word is boundedly concise in $\mathcal X$.
\end{proposition}

The proof of the above proposition uses the concept of ultraproducts. The details concerning this construction can be found, for example, in \cite{fernandez-morigi}. The following result is Lemma A.5 in \cite{fernandez-morigi}.

\begin{lemma}
\label{cardinality of the image}
Let $\GG=\{G_i\}_{i\in\N}$ be a family of groups, and for every $i\in\N$, let $S_i$ be a non-empty finite subset of $G_i$. 
If $\UU$ is an ultrafilter over $\N$, then the cardinality of the image $\bar S$ of $S=\prod_{i\in\N}\, S_i$ in the ultraproduct $\GG_{\UU}$ is given by
\begin{equation}
\label{cardinality S bar}
|\overline S| = \sup_{J\in\,\UU} \, \Big( \min_{i\in J}\, |S_i|\Big),
\end{equation}
provided that the supremum is finite, and $\overline S$ is infinite otherwise.
In particular:
\begin{enumerate}
\item
If $|S_i|\le k$ for all $i$, then $|\overline S|\le k$.
\item If the ultrafilter $\UU$ is non-principal and $|S_i|\ge k$ for big enough $i$, then $|\overline S|\ge k$.
\end{enumerate}
\end{lemma}
The next lemma shows that the class $\mathcal X$ is closed under taking  ultraproducts of its members.
\begin{lemma}\label{Xultraproducts} Let $\mathcal X$ be as in Proposition \ref{quantitative}. If $\GG=\{G_i\}_{i\in \N}$ is a family of groups in $\mathcal X$ and $\UU$ is an ultrafilter over $\N$, then the ultraproduct $\GG_{\UU}$ is again in $\mathcal X$.
\end{lemma}
\begin{proof}
For each $i\in \N$ let $N_i$ be a normal subgroup of $G_i$ of finite index at most $t$ such that $N_i$ is nilpotent of class at most $c$. Then $N=\prod_{i\in \N}\,N_i$ is a normal subgroup of the Cartesian product $\prod_{i\in \N}\,G_i$. Since $N$ is nilpotent of class at most $c$, so is its image $\bar N$ in the ultraproduct $\GG_{\UU}$.
 
It remains to prove that $\bar N$ has index at most $t$ in $\GG_{\UU}$. This amounts to proving that the order of the quotient group $\GG_{\UU}/\bar N$ is at most $t$. But $\GG_{\UU}/\bar N$ is isomorphic to the ultraproduct modulo $\UU$ of the family of groups $\{G_i/N_i\}_{i\in I}$, so we can apply Lemma \ref{cardinality of the image} (1) with the sets $S_i=G_i/N_i$ and the result follows.
\end{proof}

If $w$ is a word and $S$ is a subset of $w(G)$, we say that $w$ has width $k$ over $S$, where $k$ is a natural number, if every element of $S$ can be expressed as the product of at most $k$ elements in $G_w\cup G_w^{-1}$. The next two lemmas are Lemma A.6 and Lemma A.7 in \cite{fernandez-morigi}, respectively.
\begin{lemma}
\label{omega into product}
Let $w$ be a word, and let $\{G_i\}_{i\in I}$ be a family of groups.
Suppose that $S_i\subseteq w(G_i)$ for every $i\in I$, and that the width of $w$
can be uniformly bounded over all the subsets $S_i$.
Then
\[
\prod_{i\in I}\, S_i \subseteq w\left(\prod_{i\in I} \, G_i\right).
\]
\end{lemma}

\begin{lemma}
\label{subset of bounded width}
Let $\omega$ be a word, and let $G$ be a group such that $|\omega(G)|\ge k$, where
$k$ is a positive integer.
Then, there exists a subset $S$ of $\omega(G)$ such that $|S|\ge k$ and $\omega$
has width less than $k$ over $S$.
\end{lemma}

The proof of Proposition \ref{quantitative} will now be short.

\begin{proof}[Proof of Proposition \ref{quantitative}.]
Let $w$ be any word. We need to prove that there exists a function $f:\N\rightarrow\N$ such that if $G$ is a group in $\mathcal X$ with $|G_{w}|\le m$, then $|w(G)|\le f(m)$. By way of contradiction, assume that there is a family $\{G_i\}_{i\in\N}$ of groups in $\mathcal X$ such that
$|(G_i)_{w}|\le m$ for all $i$ but nevertheless $\lim_{i\to\infty} \, |w(G_i)|=\infty$. Let us fix an arbitrary positive integer $k$. According to Lemma \ref{subset of bounded width}, if $i$ is big enough, there is a subset
$S_i$ of $w(G_i)$ such that $|S_i|\ge k$ and $w$ has width less than $k$ over $S_i$. We complete the sequence $\{S_i\}_{i\in\N}$ by choosing the first terms equal to $1$. Now, if $G=\prod_{i\in\N} \, G_i$ and $S=\prod_{i\in\N} \, S_i$, we have
\[
G_{w} = \prod_{i\in\N} \, (G_i)_{w},
\quad
\textrm{and}
\quad
S \subseteq w(G),
\]
where the last inclusion follows from Lemma \ref{omega into product}.
Consider now a non-principal ultrafilter $\UU$ over $\N$, and let $Q=\GG_{\UU}$ be the corresponding ultraproduct. By Lemma \ref{Xultraproducts} $Q$ is in $\mathcal X$. Then $Q_{w}=\overline{(G_{w})}$ and
$w(Q)=\overline{w(G)}\supseteq \overline S$. By applying Lemma \ref{cardinality of the image}, we obtain that $|Q_{w}|\le m$ and $|w(Q)|\ge k$.
Since $k$ is arbitrary, we conclude that the verbal subgroup $w(Q)$ is infinite. This is a contradiction, since $Q$ is virtually nilpotent and
\cite{TS2} says that $w(Q)$ must be finite whenever $Q_w$ is finite.
\end{proof}

\begin{proof}[Proof of Theorem \ref{rb}] Let $w$ be a word implying virtual nilpotency of finitely generated metabelian groups, 
and let $G$ be a finite group in which $w$ has at most $m$ values. It is sufficient to show that the order of $w(G)$ is bounded in terms of $w$ and $m$. Assume that $w$ involves $k$ variables. It follows that $G$ contains a subgroup $H$ that can be generated by at most $mk$ elements such that $G_w=H_w$. It is enough to show
that the order of $w(H)$ is bounded in terms of $w$ and $m$. Therefore we can work with the group $H$ in place of $G$. Hence, without loss of generality we assume that $G$ can be generated by at most $mk$ elements. As usual, $C_G(w(G))$ has $m$-bounded index in $G$. Set $N=C_G(w(G))$. Since $w$ is a law in the quotient $N/w(N)$, the theorem of Burns and Medvedev \cite[Theorem A]{bume} says that $N/w(N)$ has a normal nilpotent subgroup $M/w(N)$ such that the nilpotency class of $M/w(N)$ and the exponent of $N/M$ are bounded in terms of $w$ only. Since $G$ is $mk$-generated and $N$ has $m$-bounded index in $G$, it follows that $N$ has a $(k,m)$-bounded number of generators. So the group $N/M$ has bounded exponent and bounded number of generators. The solution of the 
restricted Burnside problem now tells us that $N/M$ has $(k,m)$-bounded order and therefore $M$ has $(k,m)$-bounded index in $G$. Observe that $w(N)$ is contained in the center of $N$. As $M/w(N)$ has bounded nilpotency class, the same holds for $M$. Therefore with appropriate choice of $c$ and $t$, depending only on $m$ and $w$, the group $G$ belongs to the class $\mathcal X$ of groups having a normal subgroup of finite index at most $t$ and nilpotency class at most $c$. An application of Proposition \ref{quantitative} now completes the proof.
\end{proof}
\bibliographystyle{amsplain}

\end{document}